\documentclass[11pt]{article}

\usepackage{geometry,color}

\usepackage{amssymb}
\usepackage{amsmath}
\usepackage{amsthm}
\usepackage{geometry,cite,enumerate,graphicx}

\newtheorem{theorem}{Theorem}
\newtheorem{remark}{Remark}

\newcommand{\eps}{\varepsilon}

\begin{document}
\title{\bf On well-posedness of variational models of charged drops}
\author{Cyrill B. Muratov\thanks{Department of Mathematical Sciences
    and Center for Applied Mathematics and Statistics, New Jersey
    Institute of Technology, Newark, NJ 07102} \and Matteo
  Novaga\thanks{Dipartimento di Matematica, Universit\`a di Pisa,
    Largo B. Pontecorvo 5, Pisa 56127, Italy}}

\maketitle

\begin{abstract}
  Electrified liquids are well known to be prone to a variety of
  interfacial instabilities that result in the onset of apparent
  interfacial singularities and liquid fragmentation. In the case of
  electrically conducting liquids, one of the basic models describing
  the equilibrium interfacial configurations and the onset of
  instability assumes the liquid to be equipotential and interprets
  those configurations as local minimizers of the energy consisting of
  the sum of the surface energy and the electrostatic energy. Here we
  show that, surprisingly, this classical geometric variational model
  is mathematically ill-posed irrespectively of the degree to which
  the liquid is electrified. Specifically, we demonstrate that an
  isolated spherical droplet is never a local minimizer, no matter how
  small is the total charge on the droplet, since the energy can
  always be lowered by a smooth, arbitrarily small distortion of the
  droplet's surface. This is in sharp contrast with the experimental
  observations that a critical amount of charge is needed in order to
  destabilize a spherical droplet. We discuss several possible
  regularization mechanisms for the considered free boundary problem
  and argue that well-posedness can be restored by the inclusion of
  the entropic effects resulting in finite screening of free charges.
\end{abstract}







\section{Introduction}
\label{sec:introduction}

Electrospray is a technique commonly used to ionize large molecules in
aqueous solution for the purposes of mass spectrometry
\cite{gaskell97}. Since the late 1980's, it has had a profound effect
on the analytical chemistry of large biological molecules and won the
2002 Nobel Prize in Chemistry to J. B. Fenn \cite{fenn-nobel}. The
electrospray technique relies on the formation of a thin steady jet of
an electrically conducting liquid upon application of a sufficiently
high voltage to a capillary tip, as first observed by Zeleny
\cite{zeleny17}. As the voltage at the tip is increased, the liquid
meniscus progressively distorts towards a conical shape, until at a
critical voltage a ``Taylor cone'' is formed, emitting a thin liquid
jet that quickly breaks into a fine mist of charged liquid droplets
\cite{taylor64,kebarle00,fernandezdelamora07}. As the resulting
droplets move through the ambient gas, the solvent slowly evaporates,
pushing the droplets towards the so-called Rayleigh limit,
corresponding to an instability of a conducting spherical drop with
respect to arbitrarily small elongations \cite{rayleigh1882}. Upon
losing stability, the droplets undergo ``Coulomb fission'' by forming
transient jets emitting tiny droplets that carry a small fraction of
the parent droplet's mass, but a substantial fraction of its total
charge \cite{fernandezdelamora96,gomez94,duft03,achtzehn05}.  This
process repeats until the solvent in the daughter droplets is
completely evaporated, and the remaining ions enter into the gas
phase. Under certain conditions, direct evaporation of charged ionic
species from the droplet is also possible
\cite{kebarle00,fernandezdelamora07}.

The onset of the phenomena described above has long been interpreted
with the help of the basic variational model that treats the
electrified liquid as a perfect conductor
\cite{rayleigh1882,taylor64,basaran89a,fernandezdelamora07,fontelos04}. To
describe the equilibrium configurations of a charged liquid droplet,
one seeks local minimizers of the following geometric energy
functional:
\begin{align}
  \label{eq:E}
  E(\Omega) = \sigma P(\Omega) + {Q^2 \over 2 C(\Omega)}, \qquad
  |\Omega| = V.
\end{align}
Here, $\Omega \subset \mathbb R^3$ is the spatial domain occupied by
the liquid, $|\Omega|$ denotes the volume of $\Omega$ fixed to the
value $V$, $ \sigma$ is the surface tension coefficient, $Q$ is the
total electric charge on the liquid droplet, $P(\Omega)$ is the
perimeter of the set $\Omega$ understood in the sense of De Giorgi:
\begin{equation}\label{aper} 
  P(\Omega) =    \sup \left\{ \int_\Omega \nabla \cdot \phi (y) \, dy:
    \, \phi \in C^1_c(\mathbb R^3;\mathbb R^3), \ |\phi|
    \leq 1 \right\},
\end{equation}  
which coincides with the surface area for regular sets
\cite{ambrosio}, and $C(\Omega)$ is the electrical capacitance of
$\Omega$, which is defined as (using the SI units)
\begin{align}
  \label{eq:C}
  C^{-1}(\Omega) = \inf_{\mu(\Omega) = 1} 
  \int_\Omega \int_{\Omega} {1 \over 4 \pi \eps_0 |x - y|} \, d \mu(x)
  \, d \mu(y), 
\end{align}
where the infimum is taken over all non-negative Borel mesures
(``charge densities'') supported on $\Omega$, and $\eps_0$ is the
permittivity of vacuum (assuming the ambient fluid has negligible
dielectric response). Note that the infimum in \eqref{eq:C} is
attained whenever $\Omega$ is compact, with the minimizing measure
concentrating on $\partial \Omega$ \cite{landkof}. Also, $C(\Omega)$
can be equivalently expressed as \cite{lieb-loss}
\begin{align}
  \label{eq:Cv}
  C(\Omega) = \eps_0 \inf_{\substack{u \in
  D^1(\mathbb R^3) \cap C(\mathbb R^3) \\ u \geq 1 \text{ in }
  \Omega}} \int_{\mathbb R^3}    |\nabla u|^2 dx,
\end{align}
and the infimum is attained when $\Omega$ is a compact set with a
sufficiently regular boundary.

The energy in \eqref{eq:E} is the sum of the surface energy associated
with the liquid-gas interface and the electrostatic self-energy of a
conducting body occupying $\Omega$ and carrying charge $Q$. It has
been widely assumed that this energy is locally minimized by a ball of
volume $V = \frac43 \pi R^3$, as long as the charge $Q$ does not
exceed the critical charge $Q_R$ given by
\begin{align}
  \label{eq:QR}
  Q_R = 8 \pi \sqrt{\eps_0 \sigma R^3}.
\end{align}
This result was obtained in the celebrated 1882 paper of Lord
Rayleigh, who performed a {\em linear stability} analysis of the
spherical droplet with respect to small non-spherical perturbations
\cite{rayleigh1882}. Similarly, the Taylor cone at the onset of jet
formation has been interpreted as a self-similar equilibrium solution
of the Euler-Lagrange equation associated with the energy in
\eqref{eq:E}, leading to the prediction of a unique opening half-angle
of about 49.3$^\circ$ \cite{taylor64}. Recent experiments to determine
the instability threshold for levitating charged drops confirm the
onset of instability at the Rayleigh limit charge $Q = Q_R$
\cite{duft02} (for earlier studies, see
\cite{doyle64,abbas67,schweizer71}), although lower threshold values
of $Q$ have also been reported in the literature
\cite{fernandezdelamora96,richardson89,taflin89,gomez94,widmann97}. The
latter could be attributed to the presence of an unstable prolate
spheroid equilibrium state bifurcating from the ball at the Rayleigh
limit, which may serve as a transition state for the spherical droplet
agitated by the motion of the surrounding gas or thermal noise
\cite{taylor64,duft02,basaran89a,fontelos04}. The agreement between
the predicted Taylor cone angle with those observed in experiments has
been found to be less satisfactory \cite{fernandezdelamora07}.

\section{Ill-posedness of the variational model of a perfectly
  conducting liquid drop}

The overall consistency of the classical model describing the
equilibrium charged droplet configurations presented in the
introduction has recently been put into question by the work of
Goldman, Novaga and Ruffini, in which it was noted that, surprisingly,
the energy in \eqref{eq:E} admits {\em neither global nor local
  minimizers} in the natural admissible classes of sets
\cite{goldman15}. Specifically, the following result was established
for minimizers of the energy in \eqref{eq:E} (in what follows, all the
physical and material constants are assumed to be fixed, leaving $V$
and $Q$ as the only free parameters).

\begin{theorem}[\cite{goldman15}, Theorems 1.1 and 1.3]
  \label{t:nongold}
  For every $V > 0$ and $Q > 0$, the following statements are true:
  \begin{enumerate}[(i)]
  \item There is no global minimizer of the energy $E$ defined in
    \eqref{eq:E} among sets of finite perimeter.
  \item The ball of volume $V$ is not a local minimizer of $E$ defined
    in  \eqref{eq:E} with respect to perturbations that are
    arbitrarily close to it in Hausdorff distance.
  \end{enumerate}
\end{theorem}

\noindent Recall that the Hausdorff distance $d_H$ bewteen sets $A$
and $B$ is defined as
\begin{align}
  \label{Hausd}
  d_H(A, B) = \max \left\{ \sup_{x \in A} \inf_{y \in B} |x - y|,
  \sup_{y \in B} \inf_{x \in A} |x - y| \right\},  
\end{align}
where $|\cdot|$ is the Euclidian distance, and measures the closeness
of their boundaries.

Non-existence of global minimizers in Theorem \ref{t:nongold} has to
do with the fact that one can construct a minimizing sequence for the
energy in \eqref{eq:E} that consists of one big ball carrying no
charge and many tiny balls carrying all the charge $Q$ off to
infinity. Furthermore, confining the support of the minimizing
sequence to a ball of slightly bigger radius than that of the original
spherical droplet, one can still produce a sequence of competitor sets
whose energy is strictly lower than that of a single ball. These sets
again consist of a single large uncharged ball and a cloud of tiny,
but highly charged balls within an arbitrarily small distance from the
original spherical droplet's surface (for details, see
\cite{goldman15}). These observations put into serious question the
validity of the conducting drop model.

An objection to the above criticisms of the classical model is that
all the competitor configurations considered in \cite{goldman15}
consist of disconnected sets. Thus, in order for such a competitor to
be realized in a physical system, highly charged tiny droplets need to
be detached from the surface of the parent droplet, leading
effectively to {\em charge evaporation}. The latter has been the
subject of many works by the modeling community (see, e.g.,
\cite{kebarle00} and references therein), and the basic finding has
been that the thermal activation barriers typically appear to be too
high for such a process to be significant. In particular, within the
continuum model governed by \eqref{eq:E}, barrier heights have been
estimated using an ansatz-based approach, predicting the prohibitively
high values in the range of several eV \cite{labowsky00}.

Nevertheless, in what follows we demonstrate that an assumption of
connectedness does not invalidate the conclusions of Theorem
\ref{t:nongold}. Namely, we show that there exist competitor sets that
are {\em homeomorphic} to a ball and are arbitrarily close to it in
Hausdorff distance that have strictly lower energy than that of a
single spherical droplet. Thus, we establish that the spherical
droplet with volume $V$ is {\em nonlinearly unstable} for all values
of $0 < Q < Q_R$, contrary to the prediction of the linear theory
\cite{rayleigh1882}, even for local perturbations that smoothly
distort a small portion of the boundary of the spherical droplet along
the normal. We note that it implies that within the classical theory
(which also ignores discreteness of charges) there is no energy
barrier to evaporate a small highly charged droplet. This result also
implies that, surprisingly, the variational problem governed by the
energy in \eqref{eq:E} is {\em ill-posed}, and thus the classical
model of conducting drops presents the picture that is physically
incomplete. Some regularizing physical mechanisms are needed at short
length scales to make it a physically valid model.

\begin{theorem}
  \label{t:non}
  For any $V > 0$ and $Q > 0$ there exists a smooth map
  $\phi_\delta : \mathbb S^2 \to (- \delta, \delta)$ such that
  if
  \begin{align}
    \label{eq:phid} 
    \Omega_{R,\delta} = \{ x \in \mathbb R^3 \ : \ |x| \leq R +
    \phi_\delta(x / |x|) \},
  \end{align}
  then $|\Omega_{R, \delta}| = V$ and $E(\Omega_{R,\delta}) < E(B_R)$,
  where $R > 0$ is such that $V = \frac43 \pi R^3$, for all
  $\delta > 0$ sufficiently small. Moreover, one can choose
  $\text{supp} \, \phi_\delta \subset B_{\delta/R}(\nu_0)$ for some
  $\nu_0 \in \mathbb S^2$.
\end{theorem}

\noindent Here $B_R$ denotes the ball of radius $R$ centered at the
origin in $\mathbb R^3$ and $B_\delta(\nu_0)$ denotes a ball of radius
$\delta$ centered at $\nu_0$ on $\mathbb S^2$. The set
$\Omega_{R,\delta}$ is the subgraph of the function
$r = R + \phi_\delta(\nu)$ in spherical coordinates. The perturbation
$\phi_\delta$ is illustrated in Fig. \ref{f:pert}(a) and has the form
of a slender axially-symmetric protrusion from the sphere, with a
small indentation around to conserve volume.

\begin{figure}
  \centering
  \includegraphics[width=8cm]{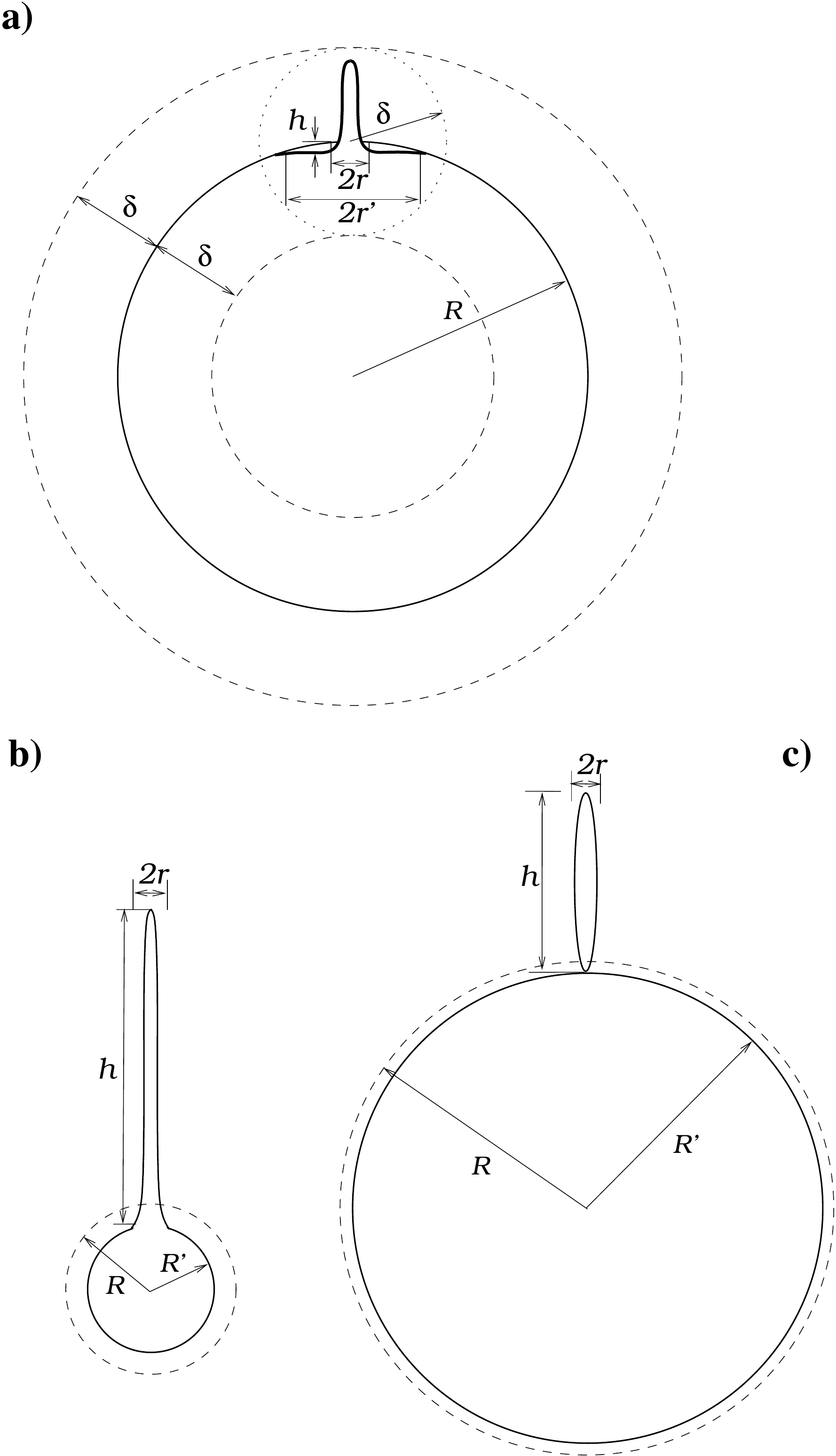}
  \caption{Different choices of perturbations of $B_R$ that may lower
    the energy in  \eqref{eq:E}: (a) a small smooth localized
    distortion of the interface; (b) a long tentacle; and (c) a
    prolate spheroid joined to a shrunk ball. }
  \label{f:pert}
\end{figure}

\begin{proof}
  Let $\eta \in C^\infty(\mathbb R)$ be a cutoff function such that
  $\eta'(t) \leq 0$ for all $t \in \mathbb R$, $\eta(t) = 1$ for
  $t \leq 1$ and $\eta(t) = 0$ for $t \geq 2$. We introduce
  \begin{align}
    \label{eq:rr}
    r = \delta e^{-R/\delta}, \qquad r' = (r R \delta^2)^{1/4},
  \end{align}
  and note that $r \ll r' \ll \delta$ for $\delta \ll R$. We then
  define the function $\phi_\delta$ in  \eqref{eq:phid} as follows:
  \begin{align}
    \label{eq:phidr}
    \phi_\delta(\nu) = \delta \eta \left( { R |\nu - \nu_0| \over r}
    \right) - h \eta \left( { 2 R |\nu - \nu_0| \over r'} \right)
    \left\{ 1 - \eta \left( { 2 R |\nu - \nu_0| \over r'} \right) \right\}, 
  \end{align}
  where $\nu_0$ is some fixed point on $\mathbb S^2$ and
  $h \sim r^2 \delta / {r'}^2$ is chosen so that
  $|\Omega_{R,\delta}| = V$ (here and in the rest of the proof the
  symbol ``$\sim$'' indicates asymptotic equivalence as $\delta \to 0$
  up to a universal positive constant, term-wise for sums). The latter
  is always possible when $\delta \ll R$, and under this assumption we
  also have $h \ll r'$.

  We next estimate from above the energy difference
  $\Delta E = E(\Omega_{R,\delta}) - E(B_R)$. In doing so, we can
  choose a suitable measure $\mu$ in the right-hand of
  \eqref{eq:C}. We take
  \begin{align}
    \label{eq:mu}
    d\mu = {1 \over 4 \pi R^2} \chi_{\partial B_R
    \backslash B_{r'}(R \nu_0)} dS + {q \over Q |C_{r,\delta}|}
    \chi_{C_{r,\delta}} dx, 
  \end{align}
  where $\chi_A$ denotes the characteristic function of the set $A$,
  $dS$ stands for the surface measure concentrated on $\partial B_R$,
  $q = Q {r'}^2/ (4 R^2) \ll Q$ is the charge in the spherical cap of
  radius $r'$ and
  $C_{r,\delta} = \{ x \in \mathbb R^3 : R < |x| < R + \delta, \ x /
  |x| \in B_{r/R}(\nu_0) \} \subset \Omega_{R,\delta}$
  is a truncated cone. Namely, we transfer the charge $q$ contained in
  the spherical cap of radius $r'$ and spread it uniformly into a set
  which for $\delta \ll R$ is essentially a straight cylinder of
  radius $r$ and height $\delta$ above $\partial B_R$ and below
  $\partial \Omega_{R,\delta}$. It is not difficult to see that with
  this configuration we have $\Delta E \leq \Delta E_0$, where
  \begin{align}
    \label{eq:dE0}
    \Delta E_0 \sim \sigma (r \delta + h^2) + {q^2 \over \eps_0 \delta} \ln
    \left( { \delta \over r} \right) - {q Q \over \eps_0 R^2} \, \delta.
  \end{align}
  With our choices of $r$ and $r'$, \eqref{eq:dE0} and \eqref{eq:QR}
  yield
  \begin{align}
    \label{eq:dE00}
    \Delta E_0 \sim {Q^2 \over \eps_0 R} \left\{ \left( {\delta Q_R^2
    \over R Q^2} + 1 \right) \left( {\delta \over R} \right)
    e^{-R/\delta} - \left( {\delta \over R} \right)^{5/2} e^{-R /
    (2\delta)} \right\}.  
  \end{align}
  Thus $\Delta E_0 < 0$ for all sufficiently small $\delta$.
\end{proof}

We also note that choosing a configuration $\Omega_{h,r}$ shown in
Fig. \ref{f:pert}(b), which consists of a smooth approximation to the
union of a ball of radius $R'$ and a long cylindrical ``tentacle'' of
radius $r \ll R$ and height $h \gg R$, provided that
$\frac43 \pi {R'}^3= \frac43 \pi R^3 - \pi r^2 h$, one can see that by
spreading the charge $Q$ over the surface of the tentacle the energy
of such a configuration can be made arbitrarily close to that of an
uncharged ball. Indeed, for $r \ll h$ one can write
\begin{align}
  \label{eq:Omhr}
  E(\Omega_{h,r}) \leq 4 \pi \sigma R^2 + 2 \pi \sigma r h + {Q^2
  \over 4 \pi \eps_0 h} \left\{ \ln \left( {2 h \over r} \right) - 1
  \right\} + o(1),
\end{align}
where we used an asymptotic formula for the capacitance of a slender
cylinder \cite{butler80}.  Therefore, optimizing this expression in
$r$ for fixed $h$, we get $r(h) = 8 Q^2 R^3 / (Q_R^2 h^2)$ and
\begin{align}
  \label{eq:EOmhrBR}
  E(\Omega_{h,r(h)}) \leq E(B_R) - {Q^2 \over 4 \pi \eps_0 h} \left\{ {h
  \over 2 R} - \ln \left( { h^3 Q_R^2 \over 4 R^3 Q^2}
  \right) \right\} + o(1).
\end{align}
One sees from these computations that for $Q < Q_R$ and
$h \gtrsim R \ln (Q_R / Q)$ we once again have
$E(\Omega_{h,r(h)}) < E(B_R)$ and, moreover,
$\lim_{h \to \infty} E(\Omega_{h,r(h)}) = 4 \pi \sigma R^2$. In
particular, this implies that global minimizers of the energy in
\eqref{eq:E} do not exist among connected sets for any $Q > 0$.

\begin{remark}
  \label{r:non}
  The proof of Theorem \ref{t:non} can be easily extended to any
  compact set $\Omega$ with smooth boundary that is a critical point
  of the energy $E$.
\end{remark}

\noindent In other words, the ill-posedness of the variational problem
associated with $E$ is not related with the specific assumption of a
spherical drop as the equilibrium configuration in Theorem
\ref{t:non}. We note that well-posedness may be restored if one seeks
minimizers in the class of sets with sufficiently smooth boundary
(say, of class $C^{1,1}$). In particular, the ball is a local
minimizer in the class of sets that satisfy the $\delta$-ball
condition (see \cite[Definition 2.18]{goldman15}) for all
$0 < Q < Q_c$, with $Q_c$ depending on $\delta > 0$ \cite[Theorem
1.4]{goldman15}. This is consistent with the linear stability result
of Rayleigh \cite{rayleigh1882}. Moreover, possibly decreasing the
value of $Q_c$, one sees that the ball is also the unique global
minimizer of $E$ in this class \cite[Theorem 1.7]{goldman15}. Still,
the restriction on the regularity of the admissible sets appears to be
rather artificial. In fact, as can be seen from Theorem \ref{t:non},
the value of $Q_c$ in the above results must go to zero as
$\delta \to 0$, which is the limit in which the regularity assumption
on the admissible sets disappears.

Alternatively, existence of minimizers for $E$ is restored in the
class of {\em convex} sets \cite{goldman15b}. Therefore, failure of
existence must necessarily be associated with the competitor sets
$\Omega$ that are concave somewhere. This is consistent with the
behavior of the Taylor cone, as the liquid typically assumes a convex
configuration prior to the formation of the cone
\cite{taylor64,fernandezdelamora07}, while convexity is broken by the
formation of the liquid jet.

\section{Well-posedness of the variational model with finite
  screening}

It is appropriate to compare the results just obtained with those for
another version of the energy in \eqref{eq:E}, namely, the {\em liquid
  drop model} of the atomic nucleus
\cite{gamow30,weizsacker35,bohr39}. From the variational standpoint,
this model is defined by the energy
\begin{align}
  \label{eq:E0}
  E_0(\Omega) = \sigma P(\Omega) + {Q^2 \over 8 \pi \eps_0  V^2}
  \int_\Omega \int_\Omega {1 \over |x - y|} \, dx \, dy, \qquad
  |\Omega| = V,
\end{align}
and can be obtained from \eqref{eq:E} by fixing the measure $\mu$ in
\eqref{eq:C} to be the uniform probability measure over $\Omega$
(uniform charge distribution). For this model (upon rescaling),
Kn\"upfer and Muratov \cite{km:cpam14} recently proved that minimizers
of $E_0$ among sets of finite perimeter exist for all $0 < Q < Q_1$,
while no minimizers exist for all $Q > Q_2 \geq Q_1$ (for
non-existence, see also \cite{lu14}). Moreover, global minimizers of
$E_0$ are balls for all $0 < Q < Q_0$ for some $Q_0 \leq Q_1$, and it
has been further conjectured that $Q_0 = Q_1 = Q_2$
\cite{choksi11}. Regardless of the validity of this conjecture, a ball
of volume $V$ is a local minimizer of the energy if and only if
$0 < Q < Q_{c2}$, where $Q_{c2}$ is the critical charge for the onset
of fission \cite{bohr39,acerbi13,julin13}.

Non-existence of minimizers for the energy in \eqref{eq:E0} has to do
with the tendency of the minimizing sequences to become disconnected
for large enough charge densities, with the connected components
moving off to infinity away from each other to minimize their mutual
Coulombic repulsion \cite{kmn:cmp}. Existence of minimizers is
restored for all $Q > 0$ when the set $\Omega$ is confined to a ball
of sufficiently large radius \cite{rigot00}. In this sense the
variational model associated with the energy in \eqref{eq:E0} is
well-posed for all charges, in sharp contrast with the one in
\eqref{eq:E}. Furthermore, minimizers of $E_0$ are regular, in the
sense that the boundary of a minimizing set consists of at most
finitely many smooth two-dimensional manifolds
\cite{rigot00,km:cpam14,julin13}.

The difficulty with the variational problem associated with the energy
in \eqref{eq:E} has to do with the freedom in the choice of the
measure $\mu$ (charge distribution) that likes to concentrate on the
singularities of the boundary of $\Omega$. In some sense, the measures
allowed by boundedness of the Coulombic energy can be more singular
than the measure associated with the essential boundary of a set of
finite perimeter, leading to a kind of incompatibility between the
two. In reality, several mechanisms may limit the ability of charges
to concentrate on the liquid interface. In the following, we show that
one such mechanism is provided by the finite screening length in the
conducting liquid.

To proceed, we need to incorporate the entropic effects associated
with the presence of free ions in the liquid. We start with the free
energy of a dilute strong electrolyte containing, for simplicity, only
two monovalent ionic species \cite{landau5,landau8}:
\begin{align}
  \label{eq:F}
  F(\Omega, n_+, n_-) = \sigma P(\Omega) + \frac{\eps_0}{2}
  \int_{\mathbb R^3} a_\Omega(x) |\nabla v|^2 dx \qquad \quad \notag
  \\ 
  + k_B T \int_\Omega
  \left( n_+ \ln {n_+ \over n_0} + n_- \ln {n_- \over n_0} \right) dx.
\end{align}
Here $a_\Omega(x) = 1 + (\eps - 1) \chi_\Omega(x)$, where
$\eps \geq 1$ is the dielectric constant of the liquid and
$\chi_\Omega$ is the characteristic function of $\Omega$, $k_B T$ is
temperature in the energy units, $n_+$ and $n_-$ are the number
densities of the positive and negative ions, respectively,
$n_0 = {1 \over 2 V} \int_\Omega (n_+ + n_-) \, dx$ is the average
free ion density per species, $v$ is the electrostatic potential
solving
\begin{align}
  \label{eq:v}
  -\eps_0 \nabla \cdot (a_\Omega(x) \nabla v) = \rho \qquad \text{in} \
  \mathcal D'(\mathbb R^3), 
\end{align}
where the charge density $\rho = e (n_+ - n_-)$ in $\Omega$ and zero
outside $\Omega$, and $e$ is the elementary charge. We next assume
that $|\rho| \ll e n_0$, i.e., that the deviations from the mean for
each ionic component are small, and expand the entropy term in $\rho$,
with $n_\pm \simeq n_0 \pm {\rho \over 2 e}$. This yields a
Debye-H\"uckel-type free energy
\begin{align}
  \label{eq:F0}
  \mathcal F(\Omega, v) = \sigma P(\Omega) + \frac{\eps_0}{2}
  \int_{\mathbb R^3} a_\Omega(x) |\nabla v|^2 dx 
  + {k_B T \over 4 e^2 n_0} \int_\Omega \rho^2 \, dx,
\end{align}
where now $\rho$ is assumed to be defined by $v$ via \eqref{eq:v}.

Since we are interested in the local well-posedness of the variational
problem governed by $\mathcal F$, we are further going to assume that
the set $\Omega$ is confined to a spherical container $B_R$ such that
$|B_R| > V$, so that the possibility of connected components of
$\Omega$ escaping to infinity is precluded. Under this assumption, we
show that minimizers of the energy $\mathcal F$ exist for all values
of $Q > 0$. More precisely, let us define an admissible class
\begin{align}
  \label{eq:A}
  \mathcal A  = \left\{ (\Omega,v):\ \Omega \subset B_R,\
  |\Omega |=V,\
  \int_\Omega \rho \, dx = Q,
  \ \rho=0 \text{ in } \mathbb R^3 \backslash \Omega
  \right\}\,,
\end{align}
where $\Omega\subset\mathbb R^3$ is a set of finite perimeter and
$v\in D^1(\mathbb R^3)$ is such that $\rho$ defined by \eqref{eq:v}
belongs to $L^2(\mathbb R^3)$. Then we have the following result.

\begin{theorem}
  \label{t:exist}
  For any $V > 0$, $Q > 0$ and $R > 0$ such that $|B_R| > V$ there
  exists a minimizer of $\mathcal F$ in $\mathcal A$.
\end{theorem}

\begin{proof}
  We follow the approach of \cite{lin93,ambrosio93} developed for
  purely dielectric problems and apply the direct method of calculus
  of variations. We consider a minimizing sequence $(\Omega_n,v_n)$
  for $\mathcal F$ in $\mathcal A$. By the energy bound, up to
  extracting a subsequence, there exist $(\Omega,v)$ such that
  $\chi_{\Omega_n}\rightharpoonup \chi_\Omega$ in $BV(\mathbb R^3)$
  and $v_n\rightharpoonup v$ in $D^1(\mathbb R^3)$.  Up to extracting
  a further subsequence, we have that the functions $\rho_n$ defined
  by \eqref{eq:v} with $v$ replaced by $v_n$ converge weakly in
  $L^2(\mathbb R^3)$ to a function $\rho$ which also solves
  \eqref{eq:v} with the limit function $v$.

  By the semicontinuity of the perimeter \cite{ambrosio} and Ioffe
  semicontinuity result \cite{ioffe77}, we get that
  \begin{align}
    \mathcal F(\Omega,v) \le \liminf_{n\to
    \infty}\mathcal F(\Omega_n,v_n)=\inf_{\mathcal A} \mathcal F.    
  \end{align}
  It remains to prove that $(\Omega,v)\in\mathcal A$.  The fact that
  $\Omega\subset B_R$ and $|\Omega|=V$ follows from the convergence of
  $\chi_{\Omega_n}$ to $\chi_\Omega$ in $L^1(\mathbb R^3)$.
  The conditions
  $\int_\Omega\rho \, dx =\int_{\mathbb R^3} \rho \chi_\Omega \, dx =
  Q$
  and $\rho=0$ on $\mathbb R^3 \backslash \Omega$, which can be
  written as $\int \rho (1-\chi_\Omega)\phi \, dx = 0$, for any
  $\phi \in \mathcal D(\mathbb R^3)$, follows from the strong
  convergence of $\chi_{\Omega_n}$ to $\chi_\Omega$ in
  $L^2(\mathbb R^3)$ and the weak convergence of $\rho_n$ to $\rho$ in
  $L^2(\mathbb R^3)$.  We thus proved that $(\Omega,v)\in\mathcal A$
  and hence is a minimizer of $\mathcal F$.
\end{proof}

This result is again in sharp contrast with the one in Theorem
\ref{t:nongold}. Further results concerning the minimizers of
$\mathcal F$ such as the regularity of their interfaces and the shape
and connectedness of minimizers are expected to follow. In particular,
in the special case of $\eps = 1$, i.e., when the dielectric
polarizability of the liquid could be neglected (or in the case of a
dielectrically matched ambient fluid), one should be able to proceed
along the lines of the arguments in \cite{km:cpam14,cicalese13} to
establish smoothness of the minimizers, as well as the fact that the
minimizers are balls for all $0 < Q < Q_c$, for some $Q_c > 0$ (even
in the absence of confinement). For $Q > Q_c$ and the confinement
radius $R$ sufficiently large one would expect the minimizers to
develop into a number of connected components of size not exceeding
much the Debye radius $r_D = \sqrt{\eps_0 \eps k_B T / (2 n_0 e^2)}$
(for a similar phenomenon in the absence of screening, see
\cite{kmn:cmp}). The reason for the latter is because for spherical
droplets whose size is much greater than $r_D$ the energy in
\eqref{eq:E} is a good approximation for the one in \eqref{eq:F0},
which then implies that one could decrease $\mathcal F$ by splitting
the droplet into many small ones and redistributing the charge, just
like in the case of the energy in \eqref{eq:E}.

\section{Discussion}

To summarize, the classical variational model of a charged liquid drop
that treats the liquid as a perfect conductor is mathematically
ill-posed. Neither global, nor local minimizers exist for this model,
with or without the presence of confinement. This is related to the
fact that the energy of a smooth charged equipotential droplet can
always be decreased by growing small, sharp protrusions that are
highly charged. Thus, any amount of spatially uncorrelated noise
should be enough to trigger the nonlinear instability identified by
us.

We note that this nonlinear instability can be related to another,
{\em linear} instability that has recently been demonstrated for thin
insulating membranes separating two different electrolytes in an
applied electric field \cite{ambjornsson07}. There it was shown that
at large enough applied voltages depending on the characteristics of
the electrolytes the membrane undergoes a long-wave instability, which
can be interpreted as a consequence of the effective membrane surface
tension coefficient turning {\em negative}. Naturally, negative
surface tension is a sign of ill-posedness of the model considered in
\cite{ambjornsson07} in the macroscopic limit. In our case, the
problem is even more singular, in the sense that the energy in
\eqref{eq:E} exhibits a {\em nonlinear} instability even for
arbitrarily small values of charge (within the continuum
approximation).

We now attempt to reconcile the differences between the models
governed by the energies in \eqref{eq:E} and \eqref{eq:F0}, as
expressed by Theorems \ref{t:nongold} and \ref{t:exist}, in connection
with the behavior of real charged drops. To this end, we obtain a
quantitative upper bound for the energy $E$ of the configurations in
the form of spherical droplets with long slender protrusions. To get
an analytical handle on the energy, we assume that the protrusion has
the shape of a prolate spheroid of radius $r$ and height $h$ that
touches a ball of radius $R' = R \sqrt[3]{1 - h r^2 / (2 R^3)}$, which
is chosen so as to conserve the volume, starting from a spherical
droplet of radius $R$. See Fig. \ref{f:pert}(c) for an
illustration. Note that we do not make any assumptions on the values
of $r$ and $h$, as long as $R'$ is positive. The upper bound for the
energy in \eqref{eq:E} is then obtained by assuming the equilibrium
charge density (as in the definition of the capacitance) with total
charge $q$ on the surface of the spheroid and the uniform charge
density with total charge $Q - q$ on the surface of the ball.

Denoting the set above by $\Omega_{h,r}$, we have
\begin{align}
  \label{eq:EOmhr}
  E(\Omega_{h,r}) \leq E_\text{surf} + E_\text{ball} + E_\text{ell} +
  E_\text{int}, 
\end{align}
where $E_\text{surf}$, $E_\text{ball}$, $E_\text{ell}$ and
$E_\text{int},$ are the total surface energy, the electrostatic
self-energy of the ball, the electrostatic self-energy of the spheroid
and the electrostatic interaction energy between the ball and the
spheroid, respectively. Because of our assumptions on the shapes, all
but the last quantity above have closed form expressions (for the
capacitance of the spheroid, see, e.g., \cite{landkof}, the other
expressions are elementary):
\begin{align}
  \label{eq:EEE}
  E_\text{surf} & = \pi  \sigma  \left(4 {R'}^2 + 2 r^2 + \frac{r h^2
                  \arcsin \sqrt{1-\frac{4 
                  r^2}{h^2}} }{\sqrt{h^2-4 r^2}} \right), \\
  E_\text{ball} & = {(Q - q)^2 \over 8 \pi \eps_0 {R'}}, \\
  E_\text{ell} & = \frac{q^2 \ln \left[  h \left(\sqrt{h^2-4 
                 r^2}+h\right) /  (2 r^2)-1 \right]
                 }{8 \pi \eps_0 \sqrt{h^2-4 r^2}}. 
\end{align}
Furthermore, the interaction energy is easily seen to be estimated
above as follows:
\begin{align}
  \label{eq:Eint}
  E_\text{int} \leq {q (Q - q) \over 8 \pi \eps_0} \left( {1 \over R'}
  + {1 \over R' + h} \right).
\end{align}
Note that for $h \ll R$ this expression gives an asymptotically sharp
value for the interaction energy.

Given the values of $r$ and $h$, we can minimize in $q$ the expression
obtained by adding up the right-hand sides of
\eqref{eq:EEE}--\eqref{eq:Eint}. The resulting cumbersome expression
may then be evaluated for any $r$ and $h$ such that
$2 r < h < 2 R^3/r^2$. In reality, we are interested in the regime
when both $r$ and $h$ are much smaller than $R$. The difference
between the obtained value and $E(B_R)$ is denoted as
$\Delta E(r, h)$, with $\Delta E(r, h) > 0$ implying that the ball of
radius $R$ has lower energy and vice versa.

We computed $\Delta E(r, h)$ as a function of $h$ for different
choices of $r$ in the case of water drops at 20$^\circ$C of radius
$R = 10\ \mu$m (a typical value from the experiments) charged to 50\%
of the Rayleigh limit, i.e., for $\sigma = 0.073$ N/m and
$Q = \frac12 Q_R$. We found that for $r$ fixed and
$r \lesssim h \lesssim R$, the value of $\Delta E(r, h)$ is initially
an increasing function of $h$ for small $h$, reaching a maximum at
$h = h_\text{max}$. After that, the value of $\Delta E(r, h)$ is a
decreasing function of $h$ and becomes negative at some
$h = h_0 > h_\text{max}$, indicating an instability. When the value of
$r$ is decreased, the value of $\Delta E(r, h_\text{max})$ decreases
as well. A natural short scale cutoff value for $r$ is $r = r_D$,
which indicates the length scale at which the assumption of the
droplet being a perfect conductor is no longer justified. For large
salt concentrations, the value of $r_D$ goes down to $r_D \simeq 1$
nm, which is also comparable to interatomic distances. Substituting
this value of $r$ to the obtained expression, we obtain the dependence
shown in Fig. \ref{f:barr}.

\begin{figure}
  \centering
  \includegraphics[width=10cm]{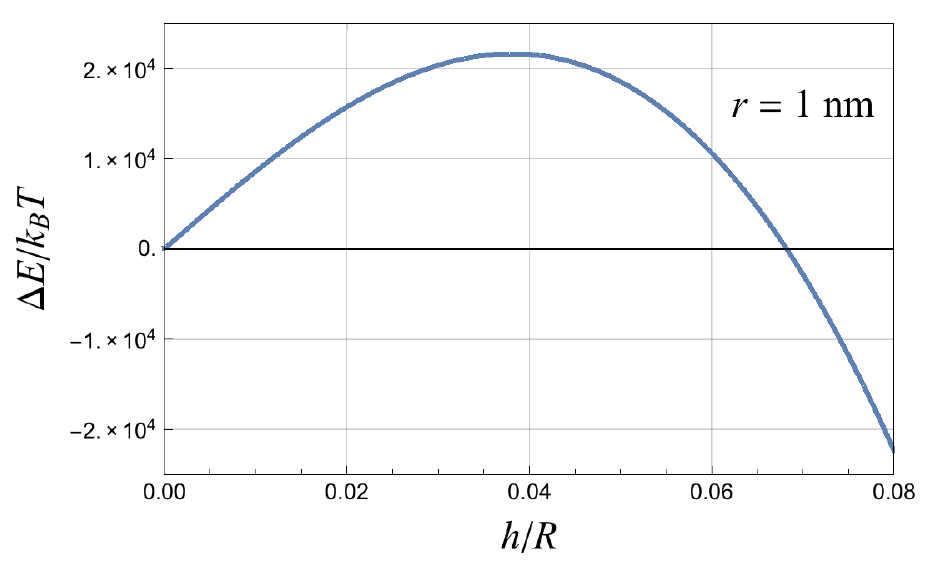}
  \caption{The dependence of $\Delta E(r, h)$ on $h$ for $r = 1$ nm in
    the case of a spherical water drop at $20^\circ$C with radius
    $R = 10\ \mu$m charged to $Q = \frac12 Q_R$.}
  \label{f:barr}
\end{figure}

From Fig. \ref{f:barr}, one can see that the maximum value of
$\Delta E(r_D, h)$ for $r_D = 1$ nm is of the order of $10^4k_BT$ and
is attained at $h = h_\text{max} \simeq 0.4 \ \mu$m. The corresponding
value of the charge in the protrusion is $q \simeq 130 e$, so the
continuum approximation is still reasonable. We interpret this value
as the energy barrier that needs to be overcome in order for a long
slender protrusion with radius of order $r_D$ to be able to grow to
decrease the energy of a spherical droplet. The value of $\Delta E$
found indicates that the barrier is prohibitively high, implying that
thermal fluctuations would not be able to excite this mode of
instability in a real system. In fact, even for methanol drops
($\sigma = 0.023$ N/m) or radius $R = 100$ nm we find that the
estimated barrier height $\Delta E$ is of the order of several hundred
$k_B T$, making it prohibitively high even in this case. Thus, we
believe that for finite screening length the nonlinear instability
discovered by us requires the level of thermal activation that may not
be readily available at room temperatures. Even if the estimate for
the barrier height above only represents the upper bound for the
actual barrier, we believe that the predicted order of magnitude for
$\Delta E$ should adequately represent the actual barrier heights.
Thus, our arguments suggest that in real systems the instability
leading to the ill-posedness of the model governed by the energy in
\eqref{eq:E} is suppressed by the low levels of thermal noise. This
might explain why the experimentally observed instability threshold
for suspended charged drops typically coincides with the threshold of
Rayleigh instability \cite{fernandezdelamora96}. We note that a
computation of the true energy barrier (within the continuum model
considered) would require a detailed numerical study of the saddle
points of the energy in \eqref{eq:F0} involving an optimization over
the shape of the protrusion and is beyond the scope of the present
study.

The conclusion above, however, needs to be taken with caution. What
our estimates suggest is that the onset of the nonlinear instability
of charged spherical drops is typically suppressed in a quiescent
medium. At the same time, in reality the drops are agitated by the
motion of the surrounding fluid and interactions with other drops, as
well as interaction with charged electrodes. Surface agitation due to
hydrodynamic effects may provide enhanced nonequilibrium fluctuations
of the interface that under certain conditions could trigger our
instability. Similarly, the presence of the external electric fields
could modify the energy barrier to make the instability driven by
thermal noise much more likely. Consider, for example, the case of a
drop as in Fig. \ref{f:barr} placed in an external field with
magnitude $|\mathbf E| = 3 \times 10^5$ V/m. Since this value is two
orders of magnitude lower than the field generated by the drop itself,
the drop shape and the charge distribution will be only slightly
perturbed away from those for a perfect sphere. At the same time, the
presence of the field lowers the value of the maximum of
$\Delta E(r_D, h)$ by $\Delta E_\text{ext} \simeq -q R |\mathbf E|$
(to the leading order for small enough $r_D$ and $h$). We find that
for this magnitude of the external field the barrier height
$\Delta E(r_D, h)$ decreases by a factor of 2. Further increase in the
magnitude of $|\mathbf E|$ leads to a further decrease of the barrier
(within the assumptions of our analysis), and the barrier height may
decrease yet further in the presence of surface contaminants that are
attracted to the region of charge concentration. Therefore,
interactions with other drops, as well as interactions with charged
electrodes, may be able to trigger our instability mechanism via
thermal fluctuations at room temperature and, in particular, result in
the onset of Coulombic fission below the Rayleigh limit
\cite{fernandezdelamora96,taflin89,gomez94,widmann97,richardson89}. This
also strongly suggests a mechanism for the experimental observations
in \cite{hager94a,hager94b}, in which a long slender tentacle abruptly
appeared as charged droplets passed a region with a strong electric
field.

Before concluding, let us comment on several other possible mechanisms
that could regularize the variational problem for the energy in
\eqref{eq:E}. One natural candidate for such a mechanism is
discreteness of charges. It provides a short-scale cutoff equal to the
average distance between the individual ions. The importance of point
charges for the process of Coulombic fission during electrospray was
pointed out in \cite{fenn93} and further quantified in
\cite{labowsky00a}. We note that in the case of a very different
electrified liquid, namely, liquid helium, charges are known to be
able to form a Wigner crystal on the liquid-gas interface
\cite{crandall71,grimes79}. Yet another possible mechanism may involve
stabilization of the interfaces due to the effective bending rigidity
supplied by the charged boundary layer, as discussed in
\cite{ambjornsson07}.

Finally, we note that the ill-posedness of the variational problem
associated with \eqref{eq:E} raises some questions about the validity
of various numerical studies of the onset of Rayleigh instability and
the formation of Taylor cones that treat the electrified liquid as a
conductor (see, e.g.,
\cite{basaran89a,betelu06,collins08,burton11,garzon14}). The numerical
discretization naturally provides a short-scale cutoff, which is
unrelated to the one in the actual physical system. Because of the
underlying ill-posedness, however, one could expect failure of
convergence when resolving interfacial singularities as the numerical
grid is refined. Therefore, in view of our results the predictions of
such numerical studies may need to be taken with caution.
 
\section*{Funding statement}

The work of C. B. M. was partially supported by NSF via grants
DMS-0908279 and DMS-1313687. M.N. was partially supported by the
Italian CNR-GNAMPA and by the University of Pisa via grant
PRA-2015-0017.

\section*{Ethics statement}

This research does not contain human or animal subjects.

\section*{Data accessibility}

This paper does not have any supporting data. 

\section*{Competing interests}

We have no competing interests.

\section*{Authors' contributions}

Both authors contributed equally in formulating, carrying out and
writing up the resutls of this research. The final version has been
approved by both authors for publication.

\bibliographystyle{unsrt}

\bibliography{../stat,../nonlin,../mura,../bio}

\begin{thebibliography}{10}

\bibitem{gaskell97}
S.~J. Gaskell.
\newblock Electrospray: Principles and practice.
\newblock {\em J. Mass Spectrom.}, 32:677--688, 1997.

\bibitem{fenn-nobel}
{The Nobel Prize in Chemistry 2002: Information for the Public}.
\newblock {The Nobel Foundation. October 9, 2002}.

\bibitem{zeleny17}
J.~Zeleny.
\newblock Instability of electrified liquid surfaces.
\newblock {\em Phys. Rev.}, 10:1--6, 1917.

\bibitem{taylor64}
G.~I. Taylor.
\newblock Disintegration of water drops in an elecric field.
\newblock {\em Proc. Roy. Soc. A}, 280:383--397, 1964.

\bibitem{kebarle00}
P.~Kebarle and M.~Peschke.
\newblock On the mechanisms by which the charged droplets produced by
  electrospray lead to gas phase ions.
\newblock {\em Analytica Chimica Acta}, 406:11--35, 2000.

\bibitem{fernandezdelamora07}
J.~Fern\'andez de~la Mora.
\newblock The fluid dynamics of {Taylor} cones.
\newblock {\em Ann. Rev. Fluid Mech.}, 39:217--243, 2007.

\bibitem{rayleigh1882}
{Lord Rayleigh}.
\newblock On the equilibrium of liquid conducting masses charged with
  electricity.
\newblock {\em Phil. Mag.}, 14:184--186, 1882.

\bibitem{fernandezdelamora96}
J.~Fern\'andez de~la Mora.
\newblock On the outcome of the {Coulombic} fission of a charged isolated drop.
\newblock {\em J. Colloid Interface Sci.}, 178:209--218, 1996.

\bibitem{gomez94}
A.~Gomez and K.~Tang.
\newblock Charge and fission of droplets in electrostatic sprays.
\newblock {\em Phys. Fluids}, 6:404--414, 1994.

\bibitem{duft03}
D.~Duft, T.~Achtzehn, R.~M\"uller, B.~A. Huber, and T.~Leisner.
\newblock Coulomb fission: {Rayleigh} jets from levitated microdroplets.
\newblock {\em Nature}, 421:128--128, 2003.

\bibitem{achtzehn05}
T.~Achtzehn, R.~M\"uller, D.~Duft, and T.~Leisner.
\newblock The {Coulomb} instability of charged microdroplets: dynamics and
  scaling.
\newblock {\em Eur. Phys. J. D}, 34:311--313, 2005.

\bibitem{basaran89a}
O.~A. Basaran and L.~E. Scriven.
\newblock Axisymmetric shapes and stability of isolated charged drops.
\newblock {\em Phys. Fluids A}, 1:795--798, 1989.

\bibitem{fontelos04}
M.~A. Fontelos and A.~Friedman.
\newblock Symmetry-breaking bifurcations of charged drops.
\newblock {\em Arch. Rat. Mech. Anal.}, 172:267--294, 2004.

\bibitem{ambrosio}
L.~Ambrosio, N.~Fusco, and D.~Pallara.
\newblock {\em Functions of bounded variation and free discontinuity problems}.
\newblock Oxford Mathematical Monographs. The Clarendon Press, New York, 2000.

\bibitem{landkof}
N.~S. Landkof.
\newblock {\em Foundations of modern potential theory}.
\newblock Springer-Verlag, New York, 1972.

\bibitem{lieb-loss}
E.~H. Lieb and M.~Loss.
\newblock {\em Analysis}.
\newblock American Mathematical Society, Providence, RI, 2010.

\bibitem{duft02}
D.~Duft, H.~Lebius, B.~A. Huber, C.~Guet, and T.~Leisner.
\newblock Shape oscillations and stability of charged microdroplets.
\newblock {\em Phys. Rev. Lett.}, 89:084503, 2002.

\bibitem{doyle64}
A.~Doyle, D.~R. Moffett, and B.~Vonnegut.
\newblock Behavior of evaporating electrically charged droplets.
\newblock {\em J. Colloid Sci.}, 19:136--143, 1964.

\bibitem{abbas67}
M.~A. Abbas and J.~Latham.
\newblock The instability of evaporating charged drops.
\newblock {\em J. Fluid Mech.}, 30:663--670, 1967.

\bibitem{schweizer71}
J.~W. Schweizer and D.N. Hanson.
\newblock Stability limit of charged drops.
\newblock {\em J. Colloid Interface Sci.}, 35:417--423, 1971.

\bibitem{richardson89}
C.~B. Richardson, A.~L. Pigg, and R.~L. Hightower.
\newblock On the stability limit of charged droplets.
\newblock {\em Proc. R. Soc. Lond. Ser. A}, 422:319--328, 1989.

\bibitem{taflin89}
D.~C. Taflin, T.~L. Ward, and E.~J. Davis.
\newblock Electrified droplet fission and the {Rayleigh} limit.
\newblock {\em Langmuir}, 5:376--384, 1989.

\bibitem{widmann97}
J.~F. Widmann, C.~L. Aardahl, and E.~J. Davis.
\newblock Observations of non-{Rayleigh} limit explosions of electrodynamically
  levitated microdroplets.
\newblock {\em Aerosol Sci. Technol.}, 27:636--648, 1997.

\bibitem{goldman15}
M.~Goldman, M.~Novaga, and B.~Ruffini.
\newblock Existence and stability for a non-local isoperimetric model of
  charged liquid drops.
\newblock {\em Arch. Rat. Mech. Anal.}, 217:1--36, 2015.

\bibitem{labowsky00}
M.~Labowsky, J.~B. Fenn, and J.~Fernandez de~la Mora.
\newblock A continuum model for ion evaporation from a drop: effect of
  curvature and charge on ion solvation energy.
\newblock {\em Analytica Chimica Acta}, 406:105--118, 2000.

\bibitem{butler80}
C.~M. Butler.
\newblock Capacitance of a finite-length conducting cylindrical tube.
\newblock {\em J. Appl. Phys.}, 51:5607--5609, 1980.

\bibitem{goldman15b}
M.~Goldman, M.~Novaga, and B.~Ruffini.
\newblock On minimizers of an isoperimetric problem with long-range
  interactions and convexity constraint.
\newblock Preprint, 2015.

\bibitem{gamow30}
G.~Gamow.
\newblock Mass defect curve and nuclear constitution.
\newblock {\em Proc. Roy. Soc. London A}, 126:632--644, 1930.

\bibitem{weizsacker35}
C.~F. von Weizs{\"a}cker.
\newblock Zur {Theorie} der {Kernmassen}.
\newblock {\em Zeitschrift f{\"u}r Physik A}, 96:431--458, 1935.

\bibitem{bohr39}
N.~Bohr and J.~A. Wheeler.
\newblock The mechanism of nuclear fission.
\newblock {\em Phys. Rev.}, 56:426--450, 1939.

\bibitem{km:cpam14}
H.~Kn\"upfer and C.~B. Muratov.
\newblock On an isoperimetric problem with a competing non-local term. {II.
  The} general case.
\newblock {\em Commun. Pure Appl. Math.}, 67:1974--1994, 2014.

\bibitem{lu14}
J.~Lu and F.~Otto.
\newblock Nonexistence of minimizer for {Thomas-Fermi-Dirac-von Weizs\"acker}
  model.
\newblock {\em Comm. Pure Appl. Math.}, 67:1605--1617, 2014.

\bibitem{choksi11}
R.~Choksi and M.~A. Peletier.
\newblock Small volume fraction limit of the diblock copolymer problem: {II.
  Diffuse} interface functional.
\newblock {\em SIAM J. Math. Anal.}, 43:739--763, 2011.

\bibitem{acerbi13}
E.~Acerbi, N.~Fusco, and M.~Morini.
\newblock Minimality via second variation for a nonlocal isoperimetric problem.
\newblock {\em Commun. Math. Phys.}, 322:515--557, 2013.

\bibitem{julin13}
V.~Julin and G.~Pisante.
\newblock Minimality via second variation for microphase separation of diblock
  copolymer melts.
\newblock Preprint: arXiv:1301.7213, 2013.

\bibitem{kmn:cmp}
H.~Kn\"upfer, C.~B. Muratov, and M.~Novaga.
\newblock Low density phases in a uniformly charged liquid.
\newblock Preprint: arXiv:1504.05600, 2015.

\bibitem{rigot00}
S.~Rigot.
\newblock Ensembles quasi-minimaux avec contrainte de volume et
  rectifiabilit\'e uniforme.
\newblock {\em M\'emoires de la SMF, 2e s\'erie}, 82:1--104, 2000.

\bibitem{landau5}
L.~D. Landau and E.~M. Lifshits.
\newblock {\em Course of Theoretical Physics}, volume~5.
\newblock Pergamon Press, London, 1980.

\bibitem{landau8}
L.~D. Landau and E.~M. Lifshits.
\newblock {\em Course of Theoretical Physics}, volume~8.
\newblock Pergamon Press, London, 1984.

\bibitem{lin93}
F.~H. Lin.
\newblock Variational problems with free interfaces.
\newblock {\em Calc. Var. PDE}, 1:149--168, 1993.

\bibitem{ambrosio93}
L.~Ambrosio and G.~Buttazzo.
\newblock An optimal design problem with perimeter penalization.
\newblock {\em Calc. Var. PDE}, 1:55--69, 1993.

\bibitem{ioffe77}
A.~D. Ioffe.
\newblock On lower semicontinuity of integral functionals. {I}.
\newblock {\em SIAM J. Control Optim.}, 15:521--538, 1977.

\bibitem{cicalese13}
M.~Cicalese and E.~Spadaro.
\newblock Droplet minimizers of an isoperimetric problem with long-range
  interactions.
\newblock {\em Comm. Pure Appl. Math.}, 66:1298--1333, 2013.

\bibitem{ambjornsson07}
T.~Ambj\"ornsson, M.~A. Lomholt, and P.~L. Hansen.
\newblock Applying a potential across a biomembrane: Electrostatic contribution
  to the bending rigidity and membrane instability.
\newblock {\em Phys. Rev. E}, 75:051916, 2007.

\bibitem{hager94a}
D.~B. Hager, N.~J. Dovichi, J.~Klassen, and P.~Kebarle.
\newblock Droplet electrospray mass spectrometry.
\newblock {\em Anal. Chem.}, 66:3944--3949, 1994.

\bibitem{hager94b}
D.~B. Hager and N.~J. Dovichi.
\newblock Behavior of microscopic liquid droplets near a strong electrostatic
  field: Droplet electrospray.
\newblock {\em Anal. Chem.}, 66:1593--1594, 1994.

\bibitem{fenn93}
J.~B. Fenn.
\newblock Ion formation from charged droplets: roles of geometry, energy, and
  time.
\newblock {\em J. Amer. Soc. Mass Spectrom.}, 4:524--535, 1993.

\bibitem{labowsky00a}
M.~Labowsky.
\newblock Discrete charge distributions in dielectric droplets.
\newblock {\em J. Colloid Interface Sci.}, 206:19--28, 1998.

\bibitem{crandall71}
R.S. Crandall and R.~Williams.
\newblock Crystallization of electrons on the surface of liquid helium.
\newblock {\em Physics Letters A}, 34:404--405, 1971.

\bibitem{grimes79}
C.~C. Grimes and G.~Adams.
\newblock Evidence for a liquid-to-crystal phase transition in a classical,
  two-dimensional sheet of electrons.
\newblock {\em Phys. Rev. Lett.}, 42:795--798, 1979.

\bibitem{betelu06}
S.~I. Betel{\'u}, M.~A. Fontelos, U.~Kindel{\'a}n, and O.~Vantzos.
\newblock Singularities on charged viscous droplets.
\newblock {\em Physics of Fluids}, 18:051706, 2006.

\bibitem{collins08}
R.~T. Collins, J.~J. Jones, M.~T. Harris, and O.~A. Basaran.
\newblock Electrohydrodynamic tip streaming and emission of charged drops from
  liquid cones.
\newblock {\em Nature Phys.}, 4:149--154, 2008.

\bibitem{burton11}
J.~C. Burton and P.~Taborek.
\newblock Simulations of {Coulombic} fission of charged inviscid drops.
\newblock {\em Phys. Rev. Lett.}, 106:144501, 2011.

\bibitem{garzon14}
M.~Garzon, L.~J. Gray, and J.~A. Sethian.
\newblock Numerical simulations of electrostatically driven jets from
  nonviscous droplets.
\newblock {\em Phys. Rev. E}, 89:033011, 2014.

\end{thebibliography}


\end{document}